\documentclass[epsfig,latexsym,amsfonts,twoside]{article}
\usepackage{amssymb,enumerate,amsmath,amscd,graphicx}

\usepackage{tikz}
\tikzset{help lines/.style={color=blue!50,very thin}}

\def\part#1{\frac{\partial\phantom{#1}}{\partial#1}}
\newtheorem{thm}{Theorem}
\newtheorem{theorem}[thm]{Theorem}

\newenvironment{proof}{\begin{trivlist}\item[]{\bf Proof} }%
{\hfill $\Box$ \end{trivlist}}
{\end{trivlist}}
\newenvironment{remark}{\begin{trivlist}\item[]{\bf Remark} }%
{\end{trivlist}}
\newenvironment{example}{\begin{trivlist}\item[]{\bf Example} }%
{\end{trivlist}}
{\end{trivlist}}

\def\Z{\ifmmode{{\mathbb Z}}\else{${\mathbb Z}$}\fi}
\def\Q{\ifmmode{{\mathbb Q}}\else{${\mathbb Q}$}\fi}
\def\C{\ifmmode{{\mathbb C}}\else{${\mathbb C}$}\fi}
\def\P{\ifmmode{{\mathbb P}}\else{${\mathbb P}$}\fi}
\def\R{\ifmmode{{\mathbb R}}\else{${\mathbb R}$}\fi}

\def\H{\ifmmode{{\mathrm H}}\else{${\mathrm H}$}\fi}

\def\B{\ifmmode{{\mathcal B}}\else{${\mathcal B}$}\fi}
\def\E{\ifmmode{{\mathcal E}}\else{${\mathcal E}$}\fi}
\def\F{\ifmmode{{\mathcal F}}\else{${\mathcal F}$}\fi}
\def\K{\ifmmode{{\mathcal K}}\else{${\mathcal K}$}\fi}
\def\L{\ifmmode{{\mathcal L}}\else{${\mathcal L}$}\fi}
\def\M{\ifmmode{{\mathcal M}}\else{${\mathcal M}$}\fi}
\def\N{\ifmmode{{\mathcal N}}\else{${\mathcal N}$}\fi}
\def\O{\ifmmode{{\mathcal O}}\else{${\mathcal O}$}\fi}
\def\U{\ifmmode{{\mathcal U}}\else{${\mathcal U}$}\fi}
\def\V{\ifmmode{{\mathcal V}}\else{${\mathcal V}$}\fi}
\def\X{\ifmmode{{\mathcal X}}\else{${\mathcal X}$}\fi}

\def\Br{\ifmmode{{\mathrm{Br}}}\else{${\mathrm{Br}}$}\fi}
\def\OG{\ifmmode{\widetilde{\cal M}_4}\else{$\widetilde{\cal M}_4$}\fi}
\def\D{\ifmmode{{\mathcal D}^b}\else{${{\mathcal
    D}^b}$}\fi}
\def\Shah{\ifmmode{\amalg\hspace*{-3.5pt}\amalg}\else{$\amalg\hspace*{-3.5pt}\amalg$}\fi}

\begin{document}

\title{A bound on the second Betti number of hyperk{\"a}hler manifolds of complex dimension six}
\author{Justin Sawon}
\date{July, 2020}
\maketitle

\begin{abstract}
Let $M$ be an irreducible compact hyperk{\"a}hler manifold of complex dimension six. Under an assumption on the Looijenga-Lunts-Verbitsky decomposition of the cohomology of $M$, we prove that the second Betti number of $M$ is at most $23$.
\end{abstract}

\maketitle

\section{Introduction}

An irreducible hyperk{\"a}hler manifold is a Riemannian manifold of real dimension $4n$ whose holonomy is equal to $\mathrm{Sp}(n)$. The Riemannian metric will be K{\"a}hlerian with respect to an $S^2$-family of complex structures, so henceforth we will always use the {\em complex\/} dimension, $2n$. Beauville~\cite{beauville99} and Guan~\cite{guan01} independently proved that the second Betti number of an irreducible compact hyperk{\"a}hler manifold of dimension four is bounded above by $23$. The Hilbert scheme of two points on a K3 surface has second Betti number $23$, so this bound is sharp.

Looijenga and Lunts~\cite{ll97} and Verbitsky~\cite{verbitsky96} showed that the cohomology of a hyperk{\"a}hler manifold admits an action of $\mathfrak{so}(4,b_2-2)$, where $b_2$ is the second Betti number. In this article we prove that the second Betti number of an irreducible compact hyperk{\"a}hler manifold of dimension six is also bounded above by $23$, under the assumption that only certain irreducible $\mathfrak{so}(4,b_2-2)$-representations appear in the Looijenga-Lunts-Verbitsky decomposition; see Theorem~\ref{six} for the precise statement. Up to deformation, there are currently three known examples of such manifolds: the Hilbert scheme of three points on a K3 surface, the generalized Kummer variety (see Beauville~\cite{beauville83}), and an example of O'Grady~\cite{ogrady03}. These examples have second Betti numbers $23$, $7$, and $8$, respectively, so once again our bound is sharp. Moreover, these examples all satisfy the assumption on the Looijenga-Lunts-Verbitsky decomposition of the cohomology.

Using the same ideas, we prove that the second Betti number of an irreducible compact hyperk{\"a}hler manifold of dimension eight is bounded above by $24$, once again under an assumption on the Looijenga-Lunts-Verbitsky decomposition {\em and\/} assuming that all odd Betti numbers vanish; see Theorem~\ref{eight} for the precise statement. In this dimension, the only examples currently known are the Hilbert scheme of four points on a K3 surface and the generalized Kummer variety, which have second Betti numbers $23$ and $7$, respectively. In particular, it is possible that no example exists in dimension eight with second Betti number $24$.

Why is it important to bound the second Betti number? The first Pontryagin class $p_1(M)$ determines a homogeneous polynomial of degree $2n-2$ on $\H^2(M,\Z)$, given by $\alpha\mapsto \int_M\alpha^{2n-2}p_1(M)$. Huybrechts~\cite{huybrechts03} proved that if the second integral cohomology $\H^2$ and the homogeneous polynomial of degree $2n-2$ on $\H^2$ determined by the first Pontryagin class are fixed, then up to diffeomorphism there are only finitely many irreducible compact hyperk{\"a}hler manifolds of dimension $2n$ realizing this structure. (Instead, one can fix $\H^2$ and a certain normalization $\tilde{q}$ of the Beauville-Bogomolov quadratic form on $\H^2$ and arrive at the same conclusion; see~\cite{huybrechts03}.) A universal bound on the second Betti number in dimension $2n$ would mean that there are finitely many possibilities for $\H^2$ as a $\mathbb{Z}$-module; it would then remain to bound the other data on $\H^2$, to conclude that there are finitely many diffeomorphism types of irreducible compact hyperk{\"a}hler manifolds of that dimension.

The author would like to thank Colleen Robles for pointing out an error in an earlier version of this article: we had overlooked some irreducible representations that could appear in the Looijenga-Lunts-Verbitsky decomposition of the cohomology. The author would also like to thank Nikon Kurnosov for conversations on this work, and the NSF for support through grants DMS-1206309 and DMS-1555206.

\section{Dimension four}

Let us recall how to bound the second Betti number in dimension four. Salamon~\cite{salamon96} proved that the Betti numbers of a compact hyperk{\"a}hler manifold of dimension $2n$ satisfy the relation
$$2\sum_{j=1}^{2n}(-1)^j(3j^2-n)b_{2n-j}=nb_{2n}.$$

\begin{theorem}[Beauville~\cite{beauville99}, Guan~\cite{guan01}]
Let $M$ be an irreducible compact hyperk{\"a}hler manifold of complex dimension four. Then the second Betti number $b_2$ of $M$ is at most $23$.
\end{theorem}

\begin{proof}
Irreducible hyperk{\"a}hler manifolds are simply-connected, so $b_1=0$. Therefore Salamon's relation for $n=2$ gives
$$-2b_3+20b_2+92=2b_4.$$
Verbitsky~\cite{verbitsky90} proved that $\mathrm{Sym}^k\H^2(M,\R)$ injects into $\H^{2k}(M,\R)$ for $k\leq n$. In particular, we can write
$$\H^4(M,\R)\cong \mathrm{Sym}^2\H^2(M,\R)\oplus \H^4_{\mathrm{prim}}(M,\R)$$
and
$$b_4={b_2+1 \choose 2} +b_4^{\prime},$$
where $b_4^{\prime}$ denotes the dimension of the primitive cohomology $\H^4_{\mathrm{prim}}(M,\R)$. Substituting this into Salamon's relation gives
$$-2b_3+20b_2+92=b_2(b_2+1)+2b_4^{\prime},$$
and therefore
$$-(b_2+4)(b_2-23)=-b_2^2+19b_2+92=2b_4^{\prime}+2b_3.$$
The left-hand side is negative if $b_2>23$, whereas the right-hand side is clearly non-negative. Therefore the second Betti number $b_2$ can be at most $23$.
\end{proof}

\begin{example}
Up to deformation, there are two known examples of irreducible compact hyperk{\"a}hler manifolds of dimension four: the Hilbert scheme $\mathrm{Hilb}^2S$ of two points on a K3 surface $S$ (see Fujiki~\cite{fujiki83}) and the generalized Kummer variety $K_2(A)$ of an abelian surface $A$ (see Beauville~\cite{beauville83}). Their Hodge diamonds are
$$\begin{array}{ccccccccc}
 & & & & 1 & & & & \\
 & & & 0 & & 0 & & & \\
 & & 1 & & 21 & & 1 & & \\
 & 0 & & 0 & & 0 & & 0 & \\
 1 & & 21 & & 232 & & 21 & & 1 \\
 & 0 & & 0 & & 0 & & 0 & \\
 & & 1 & & 21 & & 1 & & \\
 & & & 0 & & 0 & & & \\
 & & & & 1 & & & & \\
 \end{array}
 \qquad\qquad\mbox{and}\qquad\qquad
 \begin{array}{ccccccccc}
 & & & & 1 & & & & \\
 & & & 0 & & 0 & & & \\
 & & 1 & & 5 & & 1 & & \\
 & 0 & & 4 & & 4 & & 0 & \\
 1 & & 5 & & 96 & & 5 & & 1 \\
 & 0 & & 4 & & 4 & & 0 & \\
 & & 1 & & 5 & & 1 & & \\
 & & & 0 & & 0 & & & \\
 & & & & 1 & & & & \\
 \end{array},$$
with $b_2=23$, $b_3=0$, $b_4^{\prime}=0$, and $b_2=7$, $b_3=8$, $b_4^{\prime}=80$, respectively. In fact, it follows from the proof above that if $b_2=23$ then $b_3$ and $b_4^{\prime}$ must both vanish.
\end{example}

\section{Dimension six}

In higher dimensions, the injection $\mathrm{Sym}^k\H^2(M,\R)\hookrightarrow \H^{2k}(M,\R)$ is insufficient to produce a bound on the second Betti number. Instead we employ the following refinement.

\begin{theorem}[Looijenga and Lunts~\cite{ll97}, Verbitsky~\cite{verbitsky96}]
Let $M$ be an irreducible compact hyperk{\"a}hler manifold of dimension $2n$ with second Betti number $b_2$. Then there is an action of $\mathfrak{so}(4,b_2-2)$ on the real cohomology $\bigoplus_{k=0}^{4n}\H^k(M,\R)$, and hence of $\mathfrak{so}(b_2+2,\C)$ on the complex cohomology $\bigoplus_{k=0}^{4n}\H^k(M,\C)$.
\end{theorem}

\begin{remark}
This action is generated by Lefschetz operators: for each K{\"a}hler class $[\omega]$ the operators $L_{[\omega]}$ and $\Lambda_{[\omega]}$ generate an $\mathfrak{sl}(2,\C)$-action on the complex cohomology, and the amalgamation of all these actions yields the $\mathfrak{so}(b_2+2,\C)$-action.
\end{remark}

We can decompose $\bigoplus_{k=0}^{4n}\H^k(M,\C)$ into irreducible representations for this $\mathfrak{so}(b_2+2,\C)$-action. Their highest weights are related to Hodge bi-degrees; indeed, the Hodge diamond is the projection onto a plane of the (higher-dimensional) weight lattice of $\mathfrak{so}(b_2+2,\C)$. We can choose positive roots so that the dominant Weyl chamber projects onto the shaded octant of the Hodge diamond shown in Figure~\ref{diamond6}.

\begin{figure}[ht]
\begin{center}
\begin{tikzpicture}[scale=0.5]
\draw[fill=blue!20!white]  (0,7) -- (0,0) -- (-3.5,3.5);
\foreach \i in {0,...,6}
\foreach \j in {0,...,6}
\draw[fill] (\i-\j,6-\i-\j) circle (0.1cm);
\draw[thick,red] (-1,3) circle (0.25cm);
\draw[thick,red] (-2,2) circle (0.25cm);
\draw[thick,red] (0,2) circle (0.25cm);
\draw[thick,red] (-1,1) circle (0.25cm);
\draw[thick,red] (0,0) circle (0.25cm);
\end{tikzpicture}
\caption{The Hodge diamond in dimension six}
\label{diamond6}
\end{center}
\end{figure}

The irreducible representation with highest weight vector $1\in\H^0(M,\C)$ is precisely the subring of the cohomology generated by $\H^2(M,\C)$. In dimension six, the remainder of the cohomology comes from irreducible representations whose highest weight vectors lie in the Hodge groups that are circled in Figure~\ref{diamond6}. By considering all irreducible representations of $\mathfrak{so}(b_2+2,\C)$ (see Fulton and Harris~\cite{fh91}), and observing how their highest weights project to the Hodge diamond, we conclude that the only irreducible representations that could appear are those described in Table~\ref{modules6}. In the second column of this table the highest weights are given in terms of the fundamental weights. In the fourth column $\C^{b_2+2}$ and $\C$ denote the standard and trivial representations of $\mathfrak{so}(b_2+2,\C)$, respectively. Thus $V_k$ is the representation given by the $k$th exterior power $\Lambda^k\C^{b_2+2}$ of the standard representation of $\mathfrak{so}(b_2+2,\C)$. Not shown in the table is that when $b_2+2=2m+1$ is odd the largest exterior power $\Lambda^m\C^{b_2+2}$ has highest weight $2\omega_m$. In addition, when $b_2+2=2m$ is even the exterior power $\Lambda^{m-1}\C^{b_2+2}$ has highest weight $\omega_{m-1}+\omega_m$ while the middle degree exterior power $\Lambda^m\C^{b_2+2}$ is not irreducible; instead it decomposes into two irreducible representations of equal dimensions, $\Lambda^m\C^{b_2+2}=\Lambda^m_+\C^{b_2+2}\oplus\Lambda^m_-\C^{b_2+2}$, with highest weights $2\omega_{m-1}$ and $2\omega_m$, respectively. Note that tensor representations occur in even degrees in cohomology, while half-spin representations occur in odd degrees. The highest weights and dimensions of the latter are not needed for our arguments.

\begin{table}[ht]
\begin{center}
\begin{tabular}{|c|c|c|c|c|c|}
  \hline
   & highest weight & highest weight vector in & $\mathfrak{so}(b_2+2,\C)$-module & dimension \\
  \hline
  $U_{\bullet}$ & & $\H^{2,1}(M)$ & half-spin representations & \\
  $V_1$ & $\omega_1$ & $\H^{2,2}(M)$ & $\C^{b_2+2}$ & $b_2+2$ \\
  $V_2$ & $\omega_2$ & $\H^{3,1}(M)$ & $\Lambda^2\C^{b_2+2}$ & $b_2+2 \choose 2$ \\
  $V_3$ & $\omega_3$ & $\H^{3,1}(M)$ & $\Lambda^3\C^{b_2+2}$ & $b_2+2 \choose 3$ \\
  $V_4$ & $\omega_4$ & $\H^{3,1}(M)$ & $\Lambda^4\C^{b_2+2}$ & $b_2+2 \choose 4$ \\
  $\vdots$ & $\vdots$ & $\vdots$ & $\vdots$ & $\vdots$ \\
  $V_k$ & $\omega_k$ & $\H^{3,1}(M)$ & $\Lambda^k\C^{b_2+2}$ & $b_2+2 \choose k$ \\
  $\vdots$ & $\vdots$ & $\vdots$ & $\vdots$ & $\vdots$ \\  
  $W_{\bullet}$ & & $\H^{3,2}(M)$ & half-spin representations & \\
  $T$ & $0$ & $\H^{3,3}(M)$ & $\C$ & $1$ \\
  \hline
\end{tabular}
\end{center}
\caption{Irreducible representations of $\mathfrak{so}(b_2+2,\C)$ that could occur in the cohomology of $M$}
\label{modules6}
\end{table}

\begin{example}
We can calculate the dimensions of the weight spaces of these representations. The highest weight vector of $V_1$ lies in $\H^{2,2}(M)$. Acting on this with Lefschetz operators $L_{[\omega]}$ gives us classes in $\H^{4,2}(M)$, $\H^{3,3}(M)$, $\H^{2,4}(M)$, and $\H^{4,4}(M)$, and indeed we find that $V_1$ will sit inside the Hodge diamond in the following manner (where we have indicated the dimension of $V_1^{p,q}$ for each $p,q$)
$$\begin{array}{ccccccccccccc}
 & & & & & & 0 & & & & & & \\
 & & & & & 0 & & 0 & & & & & \\
 & & & & 0 & & 0 & & 0 & & & & \\
 & & & 0 & & 0 & & 0 & & 0 & & & \\
 & & 0 & & 0 & & 1 & & 0 & & 0 & & \\
 & 0 & & 0 & & 0 & & 0 & & 0 & & 0 & \\
0 & & 0 & & 1 & & b_2-2 & & 1 & & 0 & & 0 \\
 & 0 & & 0 & & 0 & & 0 & & 0 & & 0 & \\
 & & 0 & & 0 & & 1 & & 0 & & 0 & & \\
 & & & 0 & & 0 & & 0 & & 0 & & & \\
 & & & & 0 & & 0 & & 0 & & & & \\
 & & & & & 0 & & 0 & & & & & \\
 & & & & & & 0 & & & & & & \\
\end{array}.$$
The weight decomposition of the exterior power $V_k=\Lambda^kV_1$ for $k\geq 2$ is derived from the weight decomposition of $V_1$. Because the Hodge bi-degrees are derived from the weights, we can thereby determine the dimensions of the $V_k^{p.q}$s. For example,
$$\mathrm{dim}V_2^{2,2}=\mathrm{dim}V_1^{2,2}\mathrm{dim}V_1^{3,3}.$$
(Note that weight zero corresponds to Hodge bi-degree $(3,3)$; if we shift by $(3,3)$ then the bi-degrees would become additive.) We find that $V_2$ will sit inside the Hodge diamond as
$$\begin{array}{ccccccccccccc}
 & & & & & & 0 & & & & & & \\
 & & & & & 0 & & 0 & & & & & \\
 & & & & 0 & & 0 & & 0 & & & & \\
 & & & 0 & & 0 & & 0 & & 0 & & & \\
 & & 0 & & 1 & & b_2-2 & & 1 & & 0 & & \\
 & 0 & & 0 & & 0 & & 0 & & 0 & & 0 & \\
0 & & 0 & & b_2-2 & & {b_2-2 \choose 2}+2 & & b_2-2 & & 0 & & 0 \\
 & 0 & & 0 & & 0 & & 0 & & 0 & & 0 & \\
 & & 0 & & 1 & & b_2-2 & & 1 & & 0 & & \\
 & & & 0 & & 0 & & 0 & & 0 & & & \\
 & & & & 0 & & 0 & & 0 & & & & \\
 & & & & & 0 & & 0 & & & & & \\
 & & & & & & 0 & & & & & & \\
 \end{array},$$
$V_3$ will sit inside the Hodge diamond as 
$$\begin{array}{ccccccccccccc}
 & & & & & & 0 & & & & & & \\
 & & & & & 0 & & 0 & & & & & \\
 & & & & 0 & & 0 & & 0 & & & & \\
 & & & 0 & & 0 & & 0 & & 0 & & & \\
 & & 0 & & b_2-2 & & {b_2-2 \choose 2}+1 & & b_2-2 & & 0 & & \\
 & 0 & & 0 & & 0 & & 0 & & 0 & & 0 & \\
0 & & 0 & & {b_2-2 \choose 2}+1 & & {b_2-2 \choose 3}+2(b_2-2) & & {b_2-2 \choose 2}+1 & & 0 & & 0 \\
 & 0 & & 0 & & 0 & & 0 & & 0 & & 0 & \\
 & & 0 & & b_2-2 & & {b_2-2 \choose 2}+1 & & b_2-2 & & 0 & & \\
 & & & 0 & & 0 & & 0 & & 0 & & & \\
 & & & & 0 & & 0 & & 0 & & & & \\
 & & & & & 0 & & 0 & & & & & \\
 & & & & & & 0 & & & & & & \\
 \end{array},$$
and $V_4$ will sit inside the Hodge diamond as
$$\begin{array}{ccccccccccccc}
 & & & & & & 0 & & & & & & \\
 & & & & & 0 & & 0 & & & & & \\
 & & & & 0 & & 0 & & 0 & & & & \\
 & & & 0 & & 0 & & 0 & & 0 & & & \\
 & & 0 & & {b_2-2 \choose 2} & & {b_2-2 \choose 3}+b_2-2 & & {b_2-2 \choose 2} & & 0 & & \\
 & 0 & & 0 & & 0 & & 0 & & 0 & & 0 & \\
0 & & 0 & & {b_2-2 \choose 3}+b_2-2 & & {b_2-2 \choose 4}+2{b_2-2 \choose 2}+1 & & {b_2-2 \choose 3}+b_2-2 & & 0 & & 0 \\
 & 0 & & 0 & & 0 & & 0 & & 0 & & 0 & \\
 & & 0 & & {b_2-2 \choose 2} & & {b_2-2 \choose 3}+b_2-2 & & {b_2-2 \choose 2} & & 0 & & \\
 & & & 0 & & 0 & & 0 & & 0 & & & \\
 & & & & 0 & & 0 & & 0 & & & & \\
 & & & & & 0 & & 0 & & & & & \\
 & & & & & & 0 & & & & & & \\
 \end{array}.$$
Table~\ref{dimensions6} gives the dimensions of the intersections of these representations with $\H^4(M,\C)$ and $\H^6(M,\C)$.
\end{example}

\begin{table}[ht]
\begin{center}
\begin{tabular}{|c|c|c|c|c|}
  \hline
   & dimension & dimension of $V_{\bullet}\cap\H^4(M,\C)$ & dimension of $V_{\bullet}\cap\H^6(M,\C)$ \\
  \hline
  $V_1$ & $b_2+2$ & ${b_2 \choose 0}=1$ & ${b_2 \choose 1}=b_2$ \\
  $V_2$ & $b_2+2 \choose 2$ & ${b_2 \choose 1}=b_2$ & ${b_2 \choose 2}+{b_2 \choose 0}=\frac{b_2^2-b_2+2}{2}$ \\
  $V_3$ & $b_2+2 \choose 3$ & ${b_2 \choose 2}=\frac{b_2(b_2-1)}{2}$ & ${b_2 \choose 3}+{b_2 \choose 1}=\frac{b_2(b_2^2-3b_2+8)}{6}$ \\
  $V_4$ & $b_2+2 \choose 4$ & ${b_2 \choose 3}=\frac{b_2(b_2-1)(b_2-2)}{6}$ & ${b_2 \choose 4}+{b_2 \choose 2}=\frac{b_2(b_2-1)(b_2^2-5b_2+18)}{24}$ \\
  $\vdots$ & $\vdots$ & $\vdots$ & $\vdots$ \\
  $V_k$ & $b_2+2 \choose k$ & ${b_2 \choose k-1}$ & ${b_2 \choose k}+{b_2 \choose k-2}$ \\
  & & & \\
  \hline
\end{tabular}
\end{center}
\caption{Dimensions of $V_{\bullet}$ in degrees $4$ and $6$}
\label{dimensions6}
\end{table}

With these preliminaries out of the way, we can prove our main result.

\begin{theorem}
\label{six}
Let $M$ be an irreducible compact hyperk{\"a}hler manifold of complex dimension six. Of the possible irreducible representations of $\mathfrak{so}(b_2+2,\C)$ with highest weight vectors in $\H^{2,2}(M)$ and $\H^{3,1}(M)$ in the Looijenga-Lunts-Verbitsky decomposition of the cohomology of $M$, assume that only $V_1$, $V_2$, and $V_3$ can appear (i.e., assume that $V_4$, $V_5$, $\ldots$ do not appear). Then the second Betti number $b_2$ of $M$ is at most $23$.
\end{theorem}

\begin{proof}
When $n=3$ Salamon's relation gives
$$18b_4-48b_3+90b_2+210=3b_6.$$
Decompose the complex cohomology of $M$ into irreducible representations of $\mathfrak{so}(b_2+2,\C)$, as above. Suppose that $V_1$ occurs with multiplicity $c$, $V_2$ occurs with multiplicity $d$, $V_3$ occurs with multiplicity $e$, and the trivial representation $T=\C$ occurs with multiplicity $f$. The contributions of $V_1$, $V_2$, and $V_3$ to $\H^4(M,\C)$ are of dimensions $1$, $b_2$, and $\frac{b_2(b_2-1)}{2}$, respectively. Including $\mathrm{Sym}^2\H^2(M,\C)$ and multiplicities, we deduce that
$$b_4={b_2+1\choose 2}+c+db_2+e\left(\frac{b_2(b_2-1)}{2}\right).$$
Similarly, the contributions of $V_1$, $V_2$, $V_3$, and $T=\C$ to $\H^6(M,\C)$ are of dimensions $b_2$, $\frac{b_2^2-b_2+2}{2}$, $\frac{b_2(b_2^2-3b_2+8)}{6}$, and $1$, respectively. Including $\mathrm{Sym}^3\H^2(M,\C)$ and multiplicities, we deduce that
$$b_6={b_2+2\choose 3}+cb_2+d\left(\frac{b_2^2-b_2+2}{2}\right)+e\left(\frac{b_2(b_2^2-3b_2+8)}{6}\right)+f.$$
Substituting the formulae for $b_4$ and $b_6$ into Salamon's relation (and multiplying by $2$) gives
$$36\left({b_2+1\choose 2}+c+db_2+e\left(\frac{b_2(b_2-1)}{2}\right)\right)-96b_3+180b_2+420\hspace*{55mm}$$
$$\hspace*{35mm}=6b_6=6\left({b_2+2\choose 3}+cb_2+d\left(\frac{b_2^2-b_2+2}{2}\right)+e\left(\frac{b_2(b_2^2-3b_2+8)}{6}\right)+f\right),$$
and after simplifying and rearranging we obtain
$$-(b_2+6)\left(b_2-\frac{21+\sqrt{721}}{2}\right)\left(b_2-\frac{21-\sqrt{721}}{2}\right) = -b_2^3+15b_2^2+196b_2+420\hspace*{35mm}$$
$$\hspace*{50mm}=6c(b_2-6)+3d(b_2^2-13b_2+2)+eb_2(b_2^2-21b_2+26)+6f+96b_3.$$
The left-hand side is negative if $b_2\geq 24>\frac{21+\sqrt{721}}{2}\approx 23.9257$. On the other hand, $c$, $d$, $e$, $f$, and $b_3$ are all non-negative, so the right-hand side will be non-negative for $b_2\geq 24$ (indeed, for $b_2\geq 20$). Therefore the second Betti number $b_2$ can be at most $23$.
\end{proof}

\begin{remark}
The contributions of $V_4$ to $\H^4(M,\C)$ and $\H^6(M,\C)$ have dimensions
$$\frac{b_2(b_2-1)(b_2-2)}{6}\qquad\mbox{and}\qquad\frac{b_2(b_2-1)(b_2^2-5b_2+18)}{24},$$
respectively. For each occurrence of $V_4$ in the decomposition of the cohomology of $M$ we would need to add an additional term
$$6\frac{b_2(b_2-1)(b_2^2-5b_2+18)}{24}-36\frac{b_2(b_2-1)(b_2-2)}{6}=\frac{b_2(b_2-1)(b_2^2-29b_2+66)}{4}$$
to the right-hand side of the last displayed equation of the proof. If $b_2=24$, $25$, or $26$ then this term would be negative and the proof would break down. However, we could still conclude that $b_2\geq 27$ is impossible.

More generally, the contributions of $V_k$ to $\H^4(M,\C)$ and $\H^6(M,\C)$ have dimensions
$${b_2 \choose k-1}\qquad\mbox{and}\qquad {b_2 \choose k}+{b_2 \choose k-2},$$
respectively. For each occurrence of $V_k$ in the decomposition of the cohomology of $M$ we would need to add an additional term
$$6{b_2 \choose k}+6{b_2 \choose k-2}-36{b_2 \choose k-1}=\frac{6b_2(b_2-1)\cdots (b_2-k+3)}{k!}(b_2^2-(8k-3)b_2+(8k^2-16k+2))$$
to the right-hand side of the last displayed equation of the proof. Calculating the roots of the quadratic factor, we see that if
$$b_2\geq \frac{8k-3+\sqrt{32k^2+16k+1}}{2}$$
then this additional term will be non-negative, and we again reach the desired contradiction. Thus allowing $V_1$, $V_2$, $\ldots$, $V_k$ to appear in the Looijenga-Lunts-Verbitsky decomposition of the cohomology of $M$, for some $k\geq 4$, we still obtain an upper bound on $b_2$, but unfortunately this bound grows roughly linearly with $k$.
\end{remark}

\begin{example}
Up to deformation, there are three known examples of irreducible compact hyperk{\"a}hler manifolds of dimension six: the Hilbert scheme $\mathrm{Hilb}^3S$ of three points on a K3 surface $S$, the generalized Kummer variety $K_3(A)$ of an abelian surface $A$ (see Beauville~\cite{beauville83}), and an example $M_6$ of O'Grady~\cite{ogrady03}. The Hodge numbers of Hilbert schemes of points on K3 surfaces and of generalized Kummer varieties were calculated by G{\"o}ttsche and Soergel~\cite{gs93}; for $\mathrm{Hilb}^3S$ and $K_3(A)$ they are
\begin{center}
\scalebox{0.85}{
$\begin{array}{ccccccccccccc}
 & & & & & & 1 & & & & & & \\
 & & & & & 0 & & 0 & & & & & \\
 & & & & 1 & & 21 & & 1 & & & & \\
 & & & 0 & & 0 & & 0 & & 0 & & & \\
 & & 1 & & 22 & & 253 & & 22 & & 1 & & \\
 & 0 & & 0 & & 0 & & 0 & & 0 & & 0 & \\
1 & & 21 & & 253 & & 2004 & & 253 & & 21 & & 1 \\
 & 0 & & 0 & & 0 & & 0 & & 0 & & 0 & \\
 & & 1 & & 22 & & 253 & & 22 & & 1 & & \\
 & & & 0 & & 0 & & 0 & & 0 & & & \\
 & & & & 1 & & 21 & & 1 & & & & \\
 & & & & & 0 & & 0 & & & & & \\
 & & & & & & 1 & & & & & & \\
 \end{array}
 \qquad\mbox{and}\qquad
 \begin{array}{ccccccccccccc}
 & & & & & & 1 & & & & & & \\
 & & & & & 0 & & 0 & & & & & \\
 & & & & 1 & & 5 & & 1 & & & & \\
 & & & 0 & & 4 & & 4 & & 0 & & & \\
 & & 1 & & 6 & & 37 & & 6 & & 1 & & \\
 & 0 & & 4 & & 24 & & 24 & & 4 & & 0 & \\
1 & & 5 & & 37 & & 372 & & 37 & & 5 & & 1 \\
 & 0 & & 4 & & 24 & & 24 & & 4 & & 0 & \\
 & & 1 & & 6 & & 37 & & 6 & & 1 & & \\
 & & & 0 & & 4 & & 4 & & 0 & & & \\
 & & & & 1 & & 5 & & 1 & & & & \\
 & & & & & 0 & & 0 & & & & & \\
 & & & & & & 1 & & & & & & \\
 \end{array},$
}
\end{center}
with $b_2=23$, $b_3=0$, $d=1$, $c=e=f=0$, and $b_2=7$, $b_3=8$, $c=16$, $d=1$, $e=0$, $f=240$, respectively. In particular, $V_4$, $V_5$, $\ldots$ do not appear in the Looijenga-Lunts-Verbitsky decomposition of the cohomology of $\mathrm{Hilb}^3S$ and $K_3(A)$. Indeed, $e=0$ for both, so $V_3$ also does not appear.

The Hodge numbers of O'Grady's example $M_6$ were calculated by Mongardi, Rapagnetta, and Sacc{\`a}~\cite{mrs18}; they are
\begin{center}
\scalebox{0.95}{
$\begin{array}{ccccccccccccc}
 & & & & & & 1 & & & & & & \\
 & & & & & 0 & & 0 & & & & & \\
 & & & & 1 & & 6 & & 1 & & & & \\
 & & & 0 & & 0 & & 0 & & 0 & & & \\
 & & 1 & & 12 & & 173 & & 12 & & 1 & & \\
 & 0 & & 0 & & 0 & & 0 & & 0 & & 0 & \\
1 & & 6 & & 173 & & 1144 & & 173 & & 6 & & 1 \\
 & 0 & & 0 & & 0 & & 0 & & 0 & & 0 & \\
 & & 1 & & 12 & & 173 & & 12 & & 1 & & \\
 & & & 0 & & 0 & & 0 & & 0 & & & \\
 & & & & 1 & & 6 & & 1 & & & & \\
 & & & & & 0 & & 0 & & & & & \\
 & & & & & & 1 & & & & & & \\
 \end{array},$
}
\end{center}
with $b_2=8$, $b_3=0$. A priori there are two different Looijenga-Lunts-Verbitsky decompositions into irreducible $\mathfrak{so}(10,\C)$-representations that could produce this Hodge diamond: either $c=115$, $d=6$, $e=0$, $f=290$ or $c=135$, $d=0$, $e=1$, $f=240$. (Note that in neither case do $V_4$, $V_5$, $\ldots$ appear.) In fact, Green, Kim, Laza, and Robles~\cite{gklr19} have determined that the latter is the correct decomposition, but this is not immediate from the representation theory and it requires geometric arguments.
\end{example}

\begin{remark}
The hypothesis of Theorem~\ref{six} that $V_4$, $V_5$, $\ldots$ do not appear in the decomposition of the cohomology of $M$ was originally introduced somewhat artificially, to make the proof work. Indeed the remark above shows that allowing $V_4$, $V_5$, $\ldots$ to appear leads to progressively weaker bounds on $b_2$. Nevertheless, the hypothesis is satisfied for all known examples in dimension six, as observed above.

A more conceptual justification of the hypothesis is provided by Green, Kim, Laza, and Robles~\cite{gklr19}, by relating it to a conjecture of Nagai~\cite{nagai08} concerning monodromy operators for one-parameter degenerations of hyperk{\"a}hler manifolds. Green, et al.\ show that Nagai's conjecture is equivalent to a certain restriction on the highest weight vectors in the Looijenga-Lunts-Verbitsky decomposition of the cohomology in even degrees, and they verify that this restriction (and hence Nagai's conjecture) holds for all known hyperk{\"a}hler manifolds, in all dimensions. They then recognize that there is a stronger restriction on the highest weight vectors that is more natural, which also holds for all known hyperk{\"a}hler manifolds. This (conjectural) stronger restriction reduces to our hypothesis in dimension six. See~\cite{gklr19} for details.
\end{remark}

\section{Higher dimensions}

When $n=4$ Salamon's relation gives
$$2b_7+16b_6-46b_5+88b_4-142b_3+208b_2+376=4b_8.$$
Thus in dimension eight, $b_7$ appears with a coefficient of the `wrong' sign, and we cannot simply imitate the proof of Theorem~\ref{six}. To proceed, we will assume that $b_7=0$. In fact, this is equivalent to assuming that {\em all\/} odd Betti numbers vanish, as the presence of cohomology in any odd degree will force $\H^7(M,\C)$ to be non-vanishing because of the $\mathfrak{so}(b_2+2,\C)$-action. 

The Hodge diamond omitting the odd cohomology is shown in Figure~\ref{diamond8}. After removing the cohomology generated by $\H^2(M,\C)$, we are left with irreducible representations whose highest weight vectors lie in the circled Hodge groups. The only irreducible representations that could appear are those described in Table~\ref{modules8}. Like in dimension six, $V_k$ denotes that $k$th exterior power $\Lambda^k\C^{b_2+2}$ of the standard representation, while $U_1$ is given by taking the 2nd symmetric power $\mathrm{Sym}^2\C^{b_2+2}$ of the standard representation and removing the trivial direct summand $\C$, and $U_k$ is given by taking the tensor product $\C^{b_2+2}\otimes\Lambda^k\C^{b_2+2}$ of the standard representation with its $k$th exterior power and removing the direct summands $\Lambda^{k+1}\C^{b_2+2}$ and $\Lambda^{k-1}\C^{b_2+2}$ (this leaves an irreducible representation with highest weight $\omega_1+\omega_k$, except in a few special cases described in the following remark).

\begin{figure}[ht]
\begin{center}
\begin{tikzpicture}[scale=0.4]
\draw[fill=blue!20!white]  (0,9) -- (0,0) -- (-4.5,4.5);
\foreach \i in {0,...,4}
\foreach \j in {0,...,4}
\draw[fill] (2*\i-2*\j,8-2*\i-2*\j) circle (0.1cm);
\foreach \i in {0,...,3}
\foreach \j in {0,...,3}
\draw[fill] (2*\i-2*\j,6-2*\i-2*\j) circle (0.1cm);
\draw[thick,red] (-2,4) circle (0.25cm);
\draw[thick,red] (0,4) circle (0.25cm);
\draw[thick,red] (-2,2) circle (0.25cm);
\draw[thick,red] (0,2) circle (0.25cm);
\draw[thick,red] (0,0) circle (0.25cm);
\end{tikzpicture}
\caption{The Hodge diamond in dimension eight}
\label{diamond8}
\end{center}
\end{figure}

\begin{table}[ht]
\begin{center}
\scalebox{0.79}{
\begin{tabular}{|c|c|c|c|c|c|}
  \hline
   & highest weight & highest weight vector in & $\mathfrak{so}(b_2+2,\C)$-module & dimension \\
  \hline
  $U_1$ & $2\omega_1$ & $\H^{2,2}(M)$ & $\mathrm{Sym}^2\C^{b_2+2}-\C$ & $\frac{(b_2+4)(b_2+1)}{2}$ \\
  $U_2$ & $\omega_1+\omega_2$ & $\H^{3,1}(M)$ & $\C^{b_2+2}\otimes\Lambda^2\C^{b_2+2}-\Lambda^3\C^{b_2+2}-\C^{b_2+2}$ & $\frac{(b_2+4)(b_2+2)b_2}{3}$ \\
  $U_3$ & $\omega_1+\omega_3$ & $\H^{3,1}(M)$ & $\C^{b_2+2}\otimes\Lambda^3\C^{b_2+2}-\Lambda^4\C^{b_2+2}-\Lambda^2\C^{b_2+2}$ & $\frac{(b_2+4)(b_2+2)(b_2+1)(b_2-1)}{8}$ \\
  $\vdots$ & $\vdots$ & $\vdots$ & $\vdots$ & $\vdots$ \\
  $U_k$ & $\omega_1+\omega_k$ & $\H^{3,1}(M)$ & $\C^{b_2+2}\otimes\Lambda^k\C^{b_2+2}-\Lambda^{k+1}\C^{b_2+2}-\Lambda^{k-1}\C^{b_2+2}$ & $(b_2+2){b_2+2 \choose k}-{b_2+2 \choose k+1}-{b_2+2 \choose k-1}$ \\
  $\vdots$ & $\vdots$ & $\vdots$ & $\vdots$ & $\vdots$ \\
  $V_1$ & $\omega_1$ & $\H^{3,3}(M)$ & $\C^{b_2+2}$ & $b_2+2$ \\
  $V_2$ & $\omega_2$ & $\H^{4,2}(M)$ & $\Lambda^2\C^{b_2+2}$ & $b_2+2 \choose 2$ \\
  $V_3$ & $\omega_3$ & $\H^{4,2}(M)$ & $\Lambda^3\C^{b_2+2}$ & $b_2+2 \choose 3$ \\
  $V_4$ & $\omega_4$ & $\H^{4,2}(M)$ & $\Lambda^4\C^{b_2+2}$ & $b_2+2 \choose 4$ \\
  $\vdots$ & $\vdots$ & $\vdots$ & $\vdots$ & $\vdots$ \\
  $V_k$ & $\omega_k$ & $\H^{4,2}(M)$ & $\Lambda^k\C^{b_2+2}$ & $b_2+2 \choose k$ \\
  $\vdots$ & $\vdots$ & $\vdots$ & $\vdots$ & $\vdots$ \\
  $T$ & $0$ & $\H^{4,4}(M)$ & $\C$ & $1$ \\
  \hline
\end{tabular}}
\end{center}
\caption{Irreducible representations of $\mathfrak{so}(b_2+2,\C)$ that that could occur in the cohomology of $M$}
\label{modules8}
\end{table}

\begin{remark}
In Theorem~\ref{eight} we will only allow irreducible representations $U_k$ and $V_k$ with $k$ small relative to $b_2$. Nevertheless, for completeness let us clarify that if $b_2+2=2m+1$ is odd then $V_m=\Lambda^m\C^{b_2+2}$ has highest weight $2\omega_m$ and $U_m$ has highest weight $\omega_1+2\omega_m$. If $b_2+2=2m$ is even then $V_{m-1}=\Lambda^{m-1}\C^{b_2+2}$ has highest weight $\omega_{m-1}+\omega_m$, $U_{m-1}$ has highest weight $\omega_1+\omega_{m-1}+\omega_m$, $\Lambda^m\C^{b_2+2}=\Lambda^m_+\C^{b_2+2}\oplus\Lambda^m_-\C^{b_2+2}$ decomposes into two irreducible representations of equal dimensions with highest weights $2\omega_{m-1}$ and $2\omega_m$, and $\C^{b_2+2}\otimes\Lambda^m\C^{b_2+2}-\Lambda^{m+1}\C^{b_2+2}-\Lambda^{m-1}\C^{b_2+2}$ decomposes into two irreducible representations of equal dimensions with highest weights $\omega_1+2\omega_{m-1}$ and $\omega_1+2\omega_m$.
\end{remark}

Table~\ref{dimensions8} gives the dimensions of the intersections of these representations with $\H^4(M,\C)$, $\H^6(M,\C)$, and $\H^8(M,\C)$. For instance, to compute these dimensions for $U_k$ we use the description
$$U_k=\C^{b_2+2}\otimes\Lambda^k\C^{b_2+2}-\Lambda^{k+1}\C^{b_2+2}-\Lambda^{k-1}\C^{b_2+2}$$
with $\C^{b_2+2}=\C\oplus\C^{b_2}\oplus\C$ graded by $-2$, $0$, and $2$. After an overall shift of $8$, this induces the required grading on $U_k$.

\begin{table}[ht]
\begin{center}
\scalebox{0.9}{
\begin{tabular}{|c|c|c|c|}
  \hline
   & dimension of $\cap\H^4(M,\C)$ & dimension of $\cap\H^6(M,\C)$ & dimension of $\cap\H^8(M,\C)$ \\
  \hline
  $U_1$ & ${b_2 \choose 0}=1$ & $b_2{b_2 \choose 0}=b_2$ & $b_2{b_2 \choose 1}-{b_2 \choose 2}=\frac{(b_2+1)b_2}{2}$ \\
  $U_2$ & ${b_2 \choose 1}=b_2$ & $b_2{b_2 \choose 1}=b_2^2$ & $b_2{b_2 \choose 2}+b_2{b_2 \choose 0}-{b_2 \choose 3}=\frac{b_2(b_2^2+2)}{3}$ \\
  $U_3$ & ${b_2 \choose 2}=\frac{b_2(b_2-1)}{2}$ & $b_2{b_2 \choose 2}=\frac{b_2^2(b_2-1)}{2}$ & $b_2{b_2 \choose 3}+b_2{b_2 \choose 1}-{b_2 \choose 4}-{b_2 \choose 0}=\frac{(b_2+1)(b_2-1)(b_2^2-2b_2+8)}{8}$ \\
  $\vdots$ & $\vdots$ & $\vdots$ & $\vdots$ \\
  $U_k$ & ${b_2 \choose k-1}$ & $b_2{b_2 \choose k-1}$ & $b_2{b_2 \choose k}+b_2{b_2 \choose k-2}-{b_2 \choose k+1}-{b_2 \choose k-3}$ \\
  & & & \\
  $V_1$ & & ${b_2 \choose 0}=1$ & ${b_2 \choose 1}=b_2$ \\
  $V_2$ & & ${b_2 \choose 1}=b_2$ & ${b_2 \choose 2}+{b_2 \choose 0}=\frac{b_2^2-b_2+2}{2}$ \\
  $V_3$ & & ${b_2 \choose 2}=\frac{b_2(b_2-1)}{2}$ & ${b_2 \choose 3}+{b_2 \choose 1}=\frac{b_2(b_2^2-3b_2+8)}{6}$ \\
  $V_4$ & & ${b_2 \choose 3}=\frac{b_2(b_2-1)(b_2-2)}{6}$ & ${b_2 \choose 4}+{b_2 \choose 2}=\frac{b_2(b_2-1)(b_2^2-5b_2+18)}{24}$ \\
  $\vdots$ & & $\vdots$ & $\vdots$ \\
  $V_k$ & & ${b_2 \choose k-1}$ & ${b_2 \choose k}+{b_2 \choose k-2}$ \\
  & & & \\
  \hline
\end{tabular}}
\end{center}
\caption{Dimensions of $U_{\bullet}$ and $V_{\bullet}$ in degrees $4$, $6$, and $8$}
\label{dimensions8}
\end{table}

\begin{example}
The representation $U_1$ is generated by a highest weight vector in
$$U_1\cap\H^4(M,\C)=U_1\cap\H^{2,2}(M)\cong\C,$$
whereas
$$U_2\cap\H^4(M,\C)\cong\C^{b_2}$$
with
$$\mathrm{dim}(U_2\cap\H^{3,1}(M))=1=\mathrm{dim}(U_2\cap\H^{1,3}(M))\qquad\mbox{and}\qquad\mathrm{dim}(U_2\cap\H^{2,2}(M))=b_2-2.$$
Similarly
$$U_3\cap\H^4(M,\C)\cong\Lambda^2\C^{b_2}$$
with
$$\mathrm{dim}(U_3\cap\H^{3,1}(M))=b_2-2=\mathrm{dim}(U_3\cap\H^{1,3}(M))\qquad\mbox{and}\qquad\mathrm{dim}(U_3\cap\H^{2,2}(M))={b_2-2 \choose 2}+1.$$
In general
$$U_k\cap\H^4(M,\C)\cong\Lambda^{k-1}\C^{b_2}$$
with
$$\mathrm{dim}(U_k\cap\H^{3,1}(M))={b_2-2 \choose k-2}=\mathrm{dim}(U_k\cap\H^{1,3}(M))\hspace*{2mm}\mbox{and}\hspace*{2mm}\mathrm{dim}(U_k\cap\H^{2,2}(M))={b_2-2 \choose k-1}+{b_2-2 \choose k-3}.$$
\end{example}

We can now prove our result in dimension eight.

\begin{theorem}
\label{eight}
Let $M$ be an irreducible compact hyperk{\"a}hler manifold of complex dimension eight whose odd Betti numbers all vanish. Consider the part of the cohomology of $M$ not generated by $\H^2(M,\C)$. Of the possible irreducible representations of $\mathfrak{so}(b_2+2,\C)$ in its Looijenga-Lunts-Verbitsky decomposition, assume that only $U_1$, $U_2$, $U_3$, $V_1$, $V_2$, $V_3$, $V_4$, $V_5$, and $T$ can appear. Then the second Betti number $b_2$ of $M$ is at most $24$.
\end{theorem}

\begin{proof}
The proof follows the same ideas as that of Theorem~\ref{six}. When $n=4$ and all odd Betti numbers vanish, Salamon's relation gives
$$16b_6+88b_4+208b_2+376=4b_8.$$
The part of the complex cohomology of $M$ generated by $\H^2(M,\C)$ contributes ${b_2+1 \choose 2}$, ${b_2+2 \choose 3}$, and ${b_2+3 \choose 4}$ to $b_4$, $b_6$, and $b_8$, respectively. Writing the remainders of these Betti numbers as $b_4^{\prime}$, $b_6^{\prime}$, and $b_8^{\prime}$, Salamon's relation becomes
\begin{eqnarray*}
-4{b_2+3 \choose 4}+16{b_2+2 \choose 3}+88{b_2+1 \choose 2}+208b_2+376 & = & 4b_8^{\prime}-16b_6^{\prime}-88b_4^{\prime} \\
 & = & 4(b_8^{\prime}-4b_6^{\prime}-22b_4^{\prime}).
\end{eqnarray*}
After simplifying and factoring, the left-hand side becomes
$$-\frac{1}{6}(b_2+3)(b_2+8)\left(b_2-\frac{21+\sqrt{817}}{2}\right)\left(b_2-\frac{21-\sqrt{817}}{2}\right),$$
which is negative if $b_2\geq 25>\frac{21+\sqrt{817}}{2}\approx 24.7916$. It remains to show that the right-hand is non-negative for $b_2\geq 25$.

First consider the contribution of $V_k$ to the right-hand side. Each occurrence of $V_k$ in the Looijenga-Lunts-Verbitsky decomposition contributes
$${b_2 \choose k}+{b_2 \choose k-2}-4{b_2 \choose k-1}=\frac{b_2(b_2-1)\cdots (b_2-k+3)}{k!}(b_2^2-(6k-3)b_2+(6k^2-12k+2))$$
to $b_8^{\prime}-4b_6^{\prime}-22b_4^{\prime}$. This contribution will be non-negative if
$$b_2\geq\frac{6k-3+\sqrt{12k^2+12k+1}}{2}.$$
In particular, if $k=5$ we require $b_2\geq 23$, whereas if $k=6$ we would require $b_2\geq 28>\frac{33+\sqrt{505}}{2}\approx 27.7361$.

Next consider the contribution of $U_k$ to the right-hand side. Each occurrence of $U_k$ in the Looijenga-Lunts-Verbitsky decomposition contributes
$$b_2{b_2 \choose k}+b_2{b_2 \choose k-2}-{b_2 \choose k+1}-{b_2 \choose k-3}-4b_2{b_2 \choose k-1}-22{b_2 \choose k-1}\hspace*{45mm}$$
$$\hspace*{25mm}=\frac{b_2(b_2-1)\cdots (b_2-k+4)(b_2-k+2)(b_2+4)}{(k+1)(k-1)!}(b_2^2-(6k+3)b_2+(6k^2-12k-16))$$
to $b_8^{\prime}-4b_6^{\prime}-22b_4^{\prime}$. This contribution will be non-negative if
$$b_2\geq\frac{6k+3+\sqrt{12k^2+84k+73}}{2}.$$
In particular, if $k=3$ we require $b_2\geq 21>\frac{21+\sqrt{433}}{2}\approx 20.9043$, whereas if $k=4$ we would require $b_2\geq 26>\frac{27+\sqrt{601}}{2}\approx 25.7577$.

Finally, the contribution of the trivial representation $T$ to $b_8^{\prime}-4b_6^{\prime}-22b_4^{\prime}$ is just $1$, so always positive.

In conclusion, if we allow the irreducible representations $U_1$, $U_2$, $U_3$, $V_1$, $V_2$, $V_3$, $V_4$, $V_5$, and $T$ to appear in the Looijenga-Lunts-Verbitsky decomposition of the complex cohomology of $M$, then for $b_2\geq 25$ the right-hand side $4(b_8^{\prime}-4b_6^{\prime}-22b_4^{\prime})$ of Salamon's relation will be non-negative, whereas the left-hand side will be negative. This contradiction proves that the second Betti number $b_2$ can be at most $24$.
\end{proof}

\begin{remark}
As the proof shows, if we allow $U_4$ to appear then we can conclude that $b_2$ can be at most $25$, and if we allow $V_6$ to appear then we can conclude that $b_2$ can be at most $27$. Allowing $U_k$ and $V_k$ to appear for larger $k$, we still obtain upper bounds on $b_2$, but these bounds grows roughly linearly with $k$.
\end{remark}

\begin{remark}
Following the same steps in higher dimensions, the pattern appears to be that in dimension $2n$, the polynomial in $b_2$ on the left-hand side has largest root $\frac{21+\sqrt{433+96n}}{2}$. Given an irreducible compact hyperk{\"a}hler manifold whose odd Betti numbers all vanish, we expect one can show that its second Betti number must satisfy $b_2\leq \frac{21+\sqrt{433+96n}}{2}$, under some restrictions on the irreducible representations appearing in the Looijenga-Lunts-Verbitsky decomposition of its cohomology. The author has not rigorously verified this, though this direction has been pursued by Kurnosov~\cite{kurnosov15}, and indeed, Kim and Laza~\cite{kl19} later identified sufficient restrictions to guarantee that a bound of this form holds. Their restrictions are again related to Nagai's conjecture; see~\cite{kl19} and Green, et al.~\cite{gklr19} for details.
\end{remark}

\begin{flushleft}
Department of Mathematics\hfill sawon@email.unc.edu\\
University of North Carolina\hfill www.unc.edu/$\sim$sawon\\
Chapel Hill NC 27599-3250\\
USA\\
\end{flushleft}

\end{document}